\documentclass[12pt]{amsart}
\usepackage{amssymb,latexsym,amsthm,mathptmx}
\usepackage{amsmath,amsfonts,amssymb}
\usepackage{graphicx}
\usepackage{color}
\usepackage{mathrsfs}

\newcommand\R{\mathbb R}
\newcommand\C{\mathbb C}

\newcommand\D{\mathbb D}

\newcommand{\F}{\frak{F}}
\newcommand{\T}{\mathbb{T}}
\newcommand{\TT}{\mathfrak{T}}

\newcommand{\im}{\operatorname{Im}}

\newcommand{\ch}{\operatorname{Ch}}

\newtheorem{theorem}{Theorem}[section]
\newtheorem{lemma}[theorem]{Lemma}
\newtheorem{cor}[theorem]{Corollary}
\newtheorem{prop}[theorem]{Proposition}

\theoremstyle{definition}
\newtheorem{definition}[theorem]{Definition}
\newtheorem{example}[theorem]{Example}

\theoremstyle{remark}

\begin{document}
\author{
Osamu~Hatori
}
\address{
Institute of Science and Technology,
Niigata University, Niigata 950-2181, Japan
}
\email{hatori@math.sc.niigata-u.ac.jp
}



\title[]
{The Mazur-Ulam property for uniform algebras}

\keywords{Tingley's problem, the Mazur-Ulam property, surjective isometries, uniform algebras, maximal convex sets, analytic functions}

\subjclass[2020]{46B04, 46B20, 46J10, 46J15
}


\begin{abstract}
We give a sufficient condition for a Banach space with which the homogeneous extension of a surjective isometry from the unit sphere of it onto another one is real-linear. The condition is satisfied by a uniform algebra and a certain extremely $C$-regular space of real-valued continuous functions.
\end{abstract}
\maketitle
\section{Introduction}\label{sec1}
In 1987 Tingley \cite{tingley}  proposed a problem if a surjective isometry between the unit spheres of Banach spaces is extended to a surjective isometry between whole spaces. 
Wang \cite{wang} seems to be the first to solve Tingley's problem between specific spaces. He dealt with $C_0(Y)$, the space of all real (resp. complex) valued continuous functions which vanish at infinity on a locally compact Hausdorff space $Y$.  Although we do not exhibit each of the literatures, a considerable number of interesting  results have shown that Tingley's problem has an affirmative answer.
No counterexample is known.
Due to \cite[p.730]{yangzhao} Ding was the first to consider Tingley's problem between different type of spaces \cite{ding2003B}.  Ding 
\cite[Corollary 2]{ding2007} in fact proved that the real Banach space of all null sequences of real numbers satisfies now we call the Mazur-Ulam property. Later
Cheng and Dong \cite{chengdong} introduced the concept of the Mazur-Ulam property. 
Following Cheng and Dong we say that a real Banach space $E$ satisfies the Mazur-Ulam property if a surjective isometry between the unit sphere of $E$ and that of any real Banach space is extended to a surjective real-linear isometry between the whole spaces. 
Tan \cite{tan2011a,tan2011b,tan2012a} showed that the space $L^p(\R)$ for $\sigma$-finite positive measure space satisfies the Mazur-Ulam property. 
In \cite{thl2013} Tan, Huang and Liu introduced the notion of generalized lush spaces and local GL spaces and proved that every local GL space satisfies the Mazur-Ulam property.
New achievements by Mori and Ozawa \cite{moriozawa} prove that the Mazur-Ulam  property  is satisfied by unital $C^*$-algebras and real von Neumann algebras. 
Cueto-Avellaneda and Peralta \cite{cp2019} proved that a complex (resp. real) Banach space of all continuous maps with the value in a complex (resp. real) Hilbert space satisfies the Mazur-Ulam property (cf. \cite{cp2020}).
The result proving that all general JBW*-triples satisfy the Mazur-Ulam property is established by 
Becerra-Guerrero, Cueto-Avellaneda, Fern\'andez-Polo and Peralta \cite{bcfp}  and 
Kalenda and Peralta \cite{kp}. 
The study
of the Mazur-Ulam property  is nowadays a
challenging subject of study (cf. \cite{banakh,cabellosanchez2019,cp2020,wh2019}). 

In this paper we say that a complex Banach space $B$ satisfies the {\em complex  Mazur-Ulam property}, emphasizing the term `complex',  if a surjective isometry  between the unit spheres of  $B$ and any complex Banach space is extended to a surjective real-linear isometry between the whole spaces. 
Jim\'enez-Vargas, Morales-Compoy, Peralta and Ram\'irez \cite[Theorems 3.8, 3.9]{jmpr2019}  probably provides  the first examples of complex Banach spaces  satisfying the complex Mazur-Ulam property (cf. \cite{peralta2019a}).
 Note that  a complex Banach space satisfies the complex Mazur-Ulam property provided that it satisfies the Mazur-Ulam property as a real Banach space (a complex Banach space is a real Banach space simultaneously). 


In \cite{HOST} we proved that a surjective isometry between the unit spheres of uniform algebras is extended to a surjective real-linear isometry between whole of the uniform algebras. In this paper we show the complex Mazur-Ulam property for uniform algebras. 
Typical examples of a uniform algebra  consist of analytic functions of one and several complex-variables such as the disk algebra, the polydisk algebra and the ball algebra. Through the Gelfand transform, the algebra of all bounded analytic functions on a certain domain is considered as a uniform algebra on the maximal ideal space. 
Hence the main result in this paper provides the first example of a Banach space of analytic functions which satisfies the complex Mazur-Ulam property. 
  For further information about uniform algebras, see \cite{br}.

\section{Is a homogeneous extension linear?}\label{sec2}
For a real or complex Banach space $B$, we denote $S(B)$ the unit sphere $\{a\in B:\|a\|=1\}$ of $B$. A maximal convex subset of $S(B)$ is denoted by $\F_B$.
Throughout the paper the map $T:S(B_1)\to S(B_2)$  always denotes a surjective isometry with respect to the metric induced by the norm, where
$B_1$ and $B_2$ are both real Banach spaces or both complex Banach spaces.  
We define the homogeneous extension $\widetilde{T}:B_1\to B_2$ of $T$ by
\begin{equation*}
\widetilde{T}(a)=
\begin{cases}
\|a\|T\left(\frac{a}{\|a\|}\right), & 0\ne a\in B_1 \\
0,& a=0.
\end{cases}
\end{equation*}
By the definition $\widetilde{T}$ is a bijection which satisfies that $\|\widetilde{T}(a)\|=\|a\|$ for every $a\in B_1$. The Tingley's problem asks if $\widetilde{T}$ is real-linear or not. In this paper we prove that $\widetilde{T}$ is real-linear for certain Banach spaces $B_1$ including uniform algebras.

It is well known that for every $F\in \F_B$ of a real or complex Banach space $B$, there exists an extreme point $p$ in the closed unit ball $B(B^*)$ of the dual space $B^*$ of $B$ such that $F=p^{-1}(1)\cap S(B)$ (cf. \cite[Lemma 3.3]{tanaka}, \cite[Lemma 3.1]{HOST}). 
Let $Q$ be the set of all extreme points $p$ in $B(B^*)$ such that $p^{-1}(1)\cap S(B)\in \F_B$. We define an equivalence relation $\sim$ in $Q$.
We write $\T=\{z\in {\mathbb{C}}:|z|=1\}$ if $B$ is a complex Banach space, where $\C$ denotes the space of all complex numbers, and $\T=\{\pm 1\}$ if $B$ is a real Banach space.

From Definition \ref{def1} through Definition \ref{def5} $B$ is a real or complex Banach space.
\begin{definition}\label{def1}
Let  $p_1,p_2\in Q$. We denote $p_1\sim p_2$ if there exits $\gamma \in \T$ such that $p_1^{-1}(1)\cap S(B)=(\gamma p_2)^{-1}(1)\cap S(B)$.
\end{definition}
Note that $\gamma p\in Q$ provided that $\gamma\in \T$ and $p\in Q$.
\begin{lemma}\label{lemma2}
The binary relation $\sim$ is an equivalence relation in $Q$.
\end{lemma}
A proof is by a routine argument and is omitted.
\begin{definition}\label{def3}
A set of all representatives with respect to the equivalence relation $\sim$ is 
 simply called a set of representatives for $\F_B$.
\end{definition}
Note that a set of representatives exists due to the choice axiom. Note also that a set of representatives for $\F_B$ is a norming family for $B$, hence it is a uniqueness set for $B$.
\begin{example}\label{examp3}
Let  
$A$ be a uniform algebra. We assume a uniform algebra as a complex Banach space here and after.
We denote the choquet boundary for $A$ by $\ch (A)$. By the Arens-Kelley theorem (cf. \cite[p.29]{fj1}) we have $Q=\{\gamma \delta_x: x\in \ch(A),\,\,\gamma \in \T\}$, where $\delta_x$ denotes the point evaluation at $x$. In this case $\{\delta_x: x\in \ch(A)\}$ is a set of representatives for $\F_A$.
\end{example}

\begin{lemma}\label{lemma4}
Let $P$ be a set of representatives. For $F\in \F_B$  there exists a unique $(p,\lambda)\in P\times \T$ such that $F=\{a\in S(B):p(a)=\lambda\}$. Conversely, for $(p,\lambda)\in P\times \T$ we have $\{a\in S(B): p(a)=\lambda\}$ is in $\F_B$.
\end{lemma}
\begin{proof}
Let $F\in \F_B$. We first prove the existence of $(p,\lambda)\in P\times \T$ satisfying the condition. 
There exists $p\in Q$ such that $F=p^{-1}(1)\cap S(B)$. By the definition, there exists $q\in P$ such that $p\sim q$. Hence there exists $\gamma \in \T$ such that $F=(\gamma q)^{-1}(1)\cap S(B)$. Letting $\lambda=\overline{\gamma}$ we have $F=\{a\in S(B):q(a)=\lambda\}$.  
We prove the uniqueness of $(q,\lambda)$. Suppose that $\{a\in S(B):q(a)=\lambda\}=\{a\in S(B):q'(a)=\lambda'\}$ for $(q', \lambda')\in P\times \T$. Then we have $q^{-1}(1)\cap S(B)=(\lambda\overline{\lambda'}q')^{-1}\cap S(B)$, that is $q\sim q'$. As $P$ is the set of all representatives with respect to the equivalence relation $\sim$ and $q,q'\in P$ we have $q=q'$. It follows that $\lambda=\lambda'$.

Conversely, let $(p,\lambda)\in P\times \T$. Then $\overline{\lambda}p\in Q$ and $\{a\in S(B):p(a)=\lambda\}=(\overline{\lambda}p)^{-1}(1)\cap S(B)\in \F_B$. 
\end{proof}
\begin{definition}\label{def5}
For $(q,\lambda)\in  Q\times \T$, we denote $F_{q,\lambda}=\{a\in S(B):q(a)=\lambda\}$. A map 
\[
I_B:\F_B\to P\times \T
\]
is defined by $I_B(F)=(p,\lambda)$ for $F=F_{p,\lambda}\in \F_B$.
\end{definition}
By Lemma \ref{lemma4} the map $I_B$ is well defined and bijective. An important theorem of Cheng, Dong and Tanaka states that a surjective isometry between the unit spheres of Banach spaces preserves  maximal convex subsets of the unit spheres. This theorem was first exhibited by Cheng and Dong in \cite[Lemma 5.1]{chengdong} and a crystal proof was given by Tanaka \cite[Lemma 3.5]{tanaka2014b}.

In the following  $P_j$ is a set of representatives for $\F_{B_j}$ for $j=1,2$. 
Due to the theorem of Cheng, Dong and Tanaka a bijection $\TT:\F_{B_1}\to \F_{B_2}$ is induced.
\begin{definition}\label{def7}
The map $\TT:\F_{B_1}\to \F_{B_2}$ is defined by $\TT(F)=T(F)$ for $F\in \F_{B_1}$. 
The map $\TT$ is well defined and bijective. Put 
\[
\Psi=I_{B_2}\circ \TT\circ I_{B_1}^{-1}:P_1\times \T \to P_2\times \T.
\]
Define two maps
\[
\phi:P_1\times \T \to P_2
\]
and
\[
\tau:P_1\times \T \to \T
\]
by 
\begin{equation}\label{psi}
\Psi(p,\lambda)=(\phi(p,\lambda), \tau(p,\lambda)), \,\,(p,\lambda)\in P_1\times \T.
\end{equation}
If $\phi(p,\lambda)=\phi(p,\lambda')$ for every $p\in P_1$ and $\lambda, \lambda'\in \T$ we simply write $\phi(p)$ instead of $\phi(p,\lambda)$ by discarding the second term $\lambda$.
\end{definition}
Rewriting the equation \eqref{psi} we have
\begin{equation}\label{fundamental}
T(F_{p,\lambda})=F_{\phi(p,\lambda),\tau(p,\lambda)},\quad (p,\lambda)\in P_1\times \T.
\end{equation}
We point out that 
\begin{equation}\label{-}
\phi(p,-\lambda)=\phi(p,\lambda), \,\,
\tau(p,-\lambda)=-\tau(p,\lambda)
\end{equation}
for every $(p,\lambda)\in P_1\times \T$.
The reason is as follows. It is well known that $T(-F)=-T(F)$ for every $F\in \F_{B_1}$ (cf. \cite[Proposition 2.3]{mori}). Hence
\begin{multline*}
F_{\phi(p,-\lambda),\tau(p,-\lambda)}=T(F_{p,-\lambda})=T(-F_{p,\lambda})=-T(F_{p,\lambda}) \\
=-F_{\phi(p,\lambda),\tau(p,\lambda)}=F_{\phi(p,\lambda),-\tau(p,\lambda)}
\end{multline*}
for every $p\in P_1$. Since the map $I_{B_2}$ is a bijection we have \eqref{-}.

Rewriting \eqref{fundamental} 
we get  a basic equation in our argument :
\begin{equation}\label{important}
\phi(p,\lambda)(T(a))=\tau(p,\lambda),\quad a\in F_{p,\lambda}.
\end{equation}
We will prove that under some condition on $B_1$, which we will exhibit explicitly later,  we have that 
\[
\phi(p,\lambda)=\phi(p,\lambda'), \quad p\in P_1
\]
for every $\lambda$ and $\lambda'$ in $\T$, and 
\[
\tau(p,\lambda)=\tau(p,1)\times
\begin{cases}
\text{$\lambda$, \quad for some $p\in P_1$}
\\
\text{$\overline{\lambda}$, \quad for other $p$'s}
\end{cases}
\]
for $\lambda\in \T$.
We get, under some condition on $B_1$, that
\begin{equation}\label{then}
\phi(p)(T(a))=\tau(p,1)\times
\begin{cases}
\text{$p(a)$, \quad for some $p\in P_1$}
\\
\text{$\overline{p(a)}$, \quad for other $p$'s}
\end{cases}
\end{equation}
for $a\in F_{p,p(a)}$. If the equation \eqref{then} holds for any $a\in S(B_1)$, without the restriction that $a\in F_{p,p(a)}$, then applying the definition of $\widetilde{T}$ we get
\begin{equation}\label{thenthen}
\phi(p)(\widetilde{T}(a))=\tau(p,1)\times
\begin{cases}
\text{$p(a)$, \quad for some $p\in P_1$}
\\
\text{$\overline{p(a)}$, \quad for other $p$'s}
\end{cases}
\end{equation}
for every $a\in B_1$, with which we infer that
\[
\phi(p)(\widetilde{T}(a+rb))=\phi(p)(\widetilde{T}(a))+\phi(p)(r\widetilde{T}(b))
\]
for every pair $a,b\in B_1$ and every real number $r$. As $\phi(P_1)=P_2$ is a norming family,  we conclude that $\widetilde{T}$ is real-linear. 
It means that we arrive at the final positive solution for Tingley's problem under some conditions on $B_1$.

\section{Hausdorff distance between the maximal convex subsets}\label{sec3}

Recall that the Hausdorff distance $d_H(K,L)$ between non-empty closed subsets $K$ and $L$ of a metric space with  metric $d(\cdot, \cdot)$ is defined by
\[
d_H(K,L)=\max\{\sup_{a\in K}d(a,L), \sup_{b\in L}d(b,K)\}.
\]

\begin{lemma}\label{lemma8}
Let $B$ be a complex Banach space and $P$ a set of representatives for $\F_B$. 
We consider $B$ as a metric space induced by the norm. For every $p\in P$ we have
\[
d_H(F_{p,\lambda}, F_{p,\lambda'})=|\lambda-\lambda'|
\]
for every $\lambda, \lambda'\in \T$. 
Suppose  further that $p'\in P$ is different from $p$. If $\F_{p,\gamma}\cap \F_{p',-\gamma'}\ne \emptyset$
 for some $\gamma,\gamma'\in \T$, then we have
\[
d_H(F_{p,\gamma},F_{p',\gamma'})=2.
\]
\end{lemma}
\begin{proof}
Let $a\in F_{p,\lambda}$. Then $\lambda'\overline{\lambda}a\in F_{p,\lambda'}$. Thus 
$d(a,F_{p,\lambda'})\le \|a-\lambda'\overline{\lambda}a\|=|\lambda-\lambda'|$.  On the other hand, for 
every $b\in F_{p,\lambda'}$ we have $|\lambda-\lambda'|=|p(a)-p(b)|\le \|a-b\|$. Hence 
$|\lambda-\lambda'|\le d(a, F_{p,\lambda'})$. Therefore $d(a,F_{p,\lambda'})=|\lambda-\lambda'|$ holds for every $a\in F_{p,\lambda}$. In the similar way we obtain that $d(b,F_{p,\lambda})=|\lambda-\lambda'|$ for every $b\in F_{p,\lambda'}$. We conclude that $d_H(F_{p,\lambda},F_{p,\lambda'})=|\lambda-\lambda'|$ for every $\lambda,\lambda'\in \T$.

Suppose that $P\ni p'\ne p$ and $\F_{p,\gamma}\cap F_{p'-\gamma'}\ne \emptyset$ for some $\gamma,\gamma'\in \T$. Let $a\in \F_{p,\gamma}\cap F_{p'-\gamma'}$. Then for every $b\in F_{p',\gamma'}$ we have 
\[
2=|-\gamma'-\gamma'|=|p'(a)-p'(b)|\le \|a-b\|\le 2.
\]
Hence $d(a, F_{p',\gamma'})=2$. It follows that $d_H(F_{p,\gamma},F_{p',\gamma'})=2$.
\end{proof}
Note that the notion of the condition of the Hausdorff distance does not depend on the choice of $P$. In fact, we can describe the condition applying the terms of $Q$; the condition of the Hausdorff distance is satisfied by $B$ if and only if 
\begin{align*}
d_H(F_{q,\lambda},F_{q',\lambda'})=
\begin{cases}
|\lambda-\gamma\lambda'|, \quad &q^{-1}(1)\cap S(B)= (\gamma q')^{-1}(1)\cap S(B) \\
2, &q \not\sim q'
\end{cases}
\end{align*}
for $q,q'\in Q$.
\begin{example}\label{examp9}
Let $A$ be a uniform algebra and $P=\{\delta_x:x\in \ch(A)\}$, where $\delta_x$ denotes the point 
evaluation at $x$. Then  $\F_{\delta_x,\lambda}\cap \F_{\delta_{x'},\lambda'}\ne \emptyset$ for any pair of different poins $x$ and $x'$ in $\ch(A)$ and any $\lambda, \lambda'\in \T$ \cite[Lemma 4.1]{HOST}. Thus a uniform algebra satisfies the condition of the Hausdorff distance.
\end{example}

\begin{lemma}\label{lemma9}
Suppose that $B_1$ satisfies the condition of the Hausdorff distance. Let $P_1$ be a set of representatives for $\F_{B_1}$.
Then we have $\phi(p,\lambda)=\phi(p,\lambda')$ for every $p\in P_1$ and $\lambda, \lambda'\in \T$. Put 
\[
P_1^+=\{p\in P_1:\tau(p,i)=i\tau(p,1)\}
\]
and
\[
P_1^-=\{p\in P_1:\tau(p,i)=\bar{i}\tau(p,1)\}.
\]
Then $P_1^+$ and $P_1^-$ are possibly empty disjoint subsets of  
$P_1$ such that $P_1^+\cup P_1^-=P_1$. Furthermore we have
\[
\tau(p,\lambda)=\lambda \tau(p,1),\quad p\in P_1^+,\lambda\in \T
\]
and
\[
\tau(p,\lambda)=\bar{\lambda}\tau(p,1),\quad p\in P_1^-, \lambda\in \T.
\]
\end{lemma}
\begin{proof}
First we prove that $\phi(p,\lambda)=\phi(p,\lambda')$ for every $p\in P_1$ and $\lambda, \lambda'\in \T$. Suppose not: there exist $p\in P_1$ and $\lambda, \lambda' \in \T$ such that $\phi(p,\lambda)\ne \phi(p,\lambda')$. We may assume that $|\lambda-\lambda'|<2$. 
(If $\phi(p,\lambda)\ne\phi(p,\lambda')$ with $\|\lambda-\lambda'|=2$, then $\phi(p,\lambda)\ne \phi(p,i\lambda)$ or $\phi(p,\lambda')\ne \phi(p, i\lambda)$.
 Replacing $i\lambda$ by $\lambda'$ in the first case, and $i\lambda$ by $\lambda$ in the later case we have $\phi(p,\lambda)\ne \phi(p,\lambda')$ with$|\lambda-\lambda'|=\sqrt{2}<2$.)  Letting $\lambda''=\lambda^2\overline{\lambda'}$ we have $|\lambda-\lambda'|=|\lambda-\lambda''|$ and $\lambda'\ne\lambda''$. Since $B_1$ satisfies the condition of the Hausdorff distance, an element $(q,\alpha)\in P_1\times\T$ which satisfies 
\begin{equation}\label{star}
d_H(F_{p,\lambda},F_{q,\alpha})=|\lambda-\lambda'|
\end{equation}
is only two elements $(p,\lambda')$ and $(p,\lambda'')$. As $\TT$ preserves the Hausdorff distance we have
\[
d_H(F_{\phi(p,\lambda),\tau(p,\lambda)},F_{\phi(p,\lambda'),\tau(p,\lambda')})
=
d_H(F_{p,\lambda},F_{p,\lambda'})=|\lambda-\lambda'|.
\]
Applying Lemma  \ref{lemma8} we have
\begin{equation}\label{*'}
d_H(F_{\phi(p,\lambda),\tau(p,\lambda)},F_{t,\beta})=|\lambda-\lambda'|
\end{equation}
if $(t,\beta)=(\phi(p,\lambda),\lambda\overline{\lambda'}\tau(p,\lambda)), (\phi(p,\lambda),\lambda\overline{\lambda''}\tau(p,\lambda)), (\phi(p,\lambda'),\tau(p,\lambda'))$. As we suppose that $\phi(p,\lambda)\ne\phi(p,\lambda')$, the number of points $(t,\beta)\in P_2\times \T$ which satisfy \eqref{*'} is at least three, while the number of points $(q,\alpha)\in P_1\times \T$ which satisfies \eqref{star} is two since $B_1$ satisfies the condition of the Hausdorff distance. 
On the other hand the numbers of $(q,\alpha)$ and $(t,\beta)$ which satisfy \eqref{star} and  \eqref{*'} respectively must coincide each other because 
$\TT$ preserves the Hausdorff distance between the maximal convex subset. We arrive at a contradiction proving that 
$\phi(p,\lambda)=\phi(p,\lambda')$ for every $p\in P_1$ and $\lambda, \lambda'\in \T$. 
In the following we simply write $\phi(p)$ instead of $\phi(p,\lambda)$ by discarding the second term. 

We prove that $\tau(p,\lambda)=\lambda\tau(p,1)$ or $\bar{\lambda}\tau(p,1)$ for $p\in P_1$ and $\lambda\in \T$. Since $\TT$ preserves the Hausdorff distance we have by Lemma \ref{lemma8} that
\begin{multline*}
|\lambda-1|=d_H(F_{p,\lambda},F_{p,1})=d_H(F_{\phi(p),\tau(p,\lambda)},F_{\phi(p),\tau(p,1)})
\\
=
|\tau(p,\lambda)-\tau(p,1)|=|\tau(p,\lambda)\overline{\tau(p,1)}-1|,
\end{multline*}
hence 
\begin{equation}\label{marumarumaru}
\text{$\tau(p,\lambda)\overline{\tau(p,1)}=\lambda$ or $\bar{\lambda}$}
\end{equation}
Thus we have $\tau(p,\lambda)=\lambda \tau(p,1)$ or $\bar{\lambda}\tau(p,1)$. Letting $\lambda=i$ we infer that $P_1^+\cup P_1^-=P_1$ and $P_1^+\cap P_1^-=\emptyset$.  We have
\[
2=d_H(F_{p,i},F_{p,-i})=d_H(F_{\phi(p),\tau(p,i)},F_{\phi(p), \tau(p,-i)})=|\tau(p,i)-\tau(p,-i)|,
\]
hence $\tau(p,-i)=-\tau(p,i)$ for every $p\in P_1$. We show that $\tau(p,\lambda)=\lambda\tau(p,1)$ for $p\in P_1^+$ and $\lambda\in \T$. 
By \eqref{marumarumaru} we may suppose that $\im\lambda\ne 0$. In the case of $\im\lambda>0$ we have
\begin{multline*}
|i-\lambda|=d_H(F_{p,i}, F_{p,\lambda})
=d_H(F_{\phi(p), \tau(p,i)},F_{\phi(p),\tau(p,\lambda)})
\\
=|\tau(p,i)-\tau(p,\lambda)|=|i\tau(p,1)-\tau(p,\lambda)|=|i-\tau(p,\lambda)\overline{\tau(p,1)}|.
\end{multline*}
Applying \eqref{marumarumaru} we infer that $\tau(p,\lambda)\overline{\tau(p,1)}=\lambda$ as $\im\lambda>0$, hence $\tau(p,\lambda)=\lambda\tau(p,1)$. 
Next we consider the case of $\im\lambda<0$. We have
\begin{multline*}
|-i-\lambda|=d_H(F_{p,-i}, F_{p,\lambda})
\\
=d_H(F_{\phi(p), \tau(p,-i)},F_{\phi(p),\tau(p,\lambda)})
=|\tau(p,-i)-\tau(p,\lambda)|,
\end{multline*}
as $\tau(p,-i)=-\tau(p,i)$ we have 
\begin{multline*}
=|-\tau(p,i)-\tau(p,\lambda)|=|-i\tau(p,1)-\tau(p,\lambda)|=|-i-\tau(p,\lambda)\overline{\tau(p,1)}|.
\end{multline*}
As $\im \lambda<0$, we infer that 
 $\tau(p,\lambda)=\lambda\tau(p,1)$ by \eqref{marumarumaru}.

Similarly, it can be proved that $\tau(p,\lambda)=\bar{\lambda}\tau(p,1)$ if $p\in P_1^-$ and $\lambda\in \T$.
\end{proof}
\section{A sufficient condition for the complex Mazur-Ulam property}\label{sec4}
We define a set $M_{p,\alpha}$ with which the map $\Psi$ plays a crucial role to work out the  complex Mazur-Ulam property for a uniform algebra. We exhibit the definition of $M_{p,\alpha}$. We denote $\bar\D=\{z\in {\mathbb{F}}:|z|\le 1\}$, where $\mathbb{F}=\mathbb{R}$ if the corresponding Banach space is a real one and 
$\mathbb{F}=\mathbb{C}$ if the corresponding Banach space is 
a complex one.
\begin{definition}\label{def12}
Let $B$ be a real or complex Banach space and $P$ a set of representatives for $\F_B$. 
For $p\in P$ and $\alpha\in \bar\D$ we denote
\[
M_{p,\alpha}=\{a\in S(B): d(a,F_{p,\alpha/|\alpha|})\le 1-|\alpha|, 
d(a, F_{p,-\alpha/|\alpha|})\le 1+|\alpha|\},
\]
where we read $\alpha/|\alpha|=1$ if $\alpha=0$. 
\end{definition}
\begin{lemma}\label{mpalpha}
If $B_j$ is a real Banach space for $j=1,2$, then we have
\[
T(M_{p,\alpha})=M_{\phi(p),\alpha\tau(p,1)}
\]
for every $(p,\alpha)\in P_1\times \T$. If $B_j$ is a  complex Banach space which satisfy the condition of the Hausdorff distance for $j=1,2$, then we have
\[
T(M_{p,\alpha})=
\begin{cases}
M_{\phi(p),\alpha \tau(p,1)},\quad p\in P_1^+
\\
M_{\phi(p),\overline{\alpha}\tau(p, 1)},\quad p\in P_1^-
\end{cases}
\]
for every $(p,\alpha)\in P_1\times \T$.
\end{lemma}
\begin{proof}
Due to the definition of the map $\Psi$ we have
\[
T(F_{p,\frac{\alpha}{|\alpha|}})=F_{\phi(p,\frac{\alpha}{|\alpha|}), \tau(p,\frac{\alpha}{|\alpha|})}
\]
and
\[
T(F_{p,-\frac{\alpha}{|\alpha|}})=F_{\phi(p,-\frac{\alpha}{|\alpha|}), \tau(p,-\frac{\alpha}{|\alpha|})}
\]

Suppose that  $B_j$ is a real Banach space first. Then by the definition $\T=\{\pm1\}$.
By \eqref{-} we have $\phi(p,1)=\phi(p,-1)$ for every $p\in P_1$. 
Hence $\phi(p,\lambda)$ does not depend on the second term for a real Banach space. We also have $\tau(p,-1)=-\tau(p,1)$ for every $p\in P_1$ by \eqref{-}. It follows that
\[
T(F_{p,\frac{\alpha}{|\alpha|}})=F_{\phi(p),\frac{\alpha}{|\alpha|}\tau(p,1)}
\]
and
\[
T(F_{p,-\frac{\alpha}{|\alpha|}})=F_{\phi(p),-\frac{\alpha}{|\alpha|}\tau(p,1)}
\]
As $T$ is a surjective isometry we have
\[
d(a,F_{p,\frac{\alpha}{|\alpha|}})=d(T(a),F_{\phi(p),\frac{\alpha}{|\alpha|}\tau(p,1)})
\]
and
\[
d(a,F_{p,-\frac{\alpha}{|\alpha|}})=d(T(a),F_{\phi(p),-\frac{\alpha}{|\alpha|}\tau(p,1)})
\]
As  $T$ is a bijection we conclude that
\[
T(M_{p,\alpha})=M_{\phi(p),\alpha\tau(p,1)}
\]
for every $p\in P_1$ and $\alpha\in \bar\D$.

Suppose next that $B_j$ is a complex Banach space which satisfies that the condition of the Hausdorff distance. Let $p\in P_1^+$. Then by Lemma \ref{lemma9} we have 
\[
T(F_{p,\frac{\alpha}{|\alpha|}})=F_{\phi(p),\frac{\alpha}{|\alpha|}\tau(p,1)}
\]
and
\[
T(F_{p,-\frac{\alpha}{|\alpha|}})=F_{\phi(p),-\frac{\alpha}{|\alpha|}\tau(p,1)}.
\]
The left of the proof is similar to the case that $B_j$ is a real Banach space, hence we see that
\[
T(M_p,\alpha)=M_{\phi(p),\alpha\tau(p,1)}
\]
for every $p\in P_1^+$ and $\alpha\in \bar\D$.  The proof for the case of $p\in P_1^-$ is similar and is omitted.
\end{proof}

\begin{lemma}\label{lemma}
Suppose that $B$ is a real or complex Banach space and $P$ is a set of representatives for $\F_B$. For every $p\in P$ and $\alpha\in \mathbb{D}$ we have
\[
M_{p,\alpha}\subset \{a\in S(B):p(a)=\alpha\}.
\]
\end{lemma} 
\begin{proof}
Let $a\in M_{p,\alpha}$. Then for arbitrary $b\in F_{p,\frac{\alpha}{|\alpha|}}$
\[
\left|p(a)-\frac{\alpha}{|\alpha|}\right|=|p(a)-p(b)|\le \|a-b\|.
\]
As $d(a,F_{p,\frac{\alpha}{|\alpha|}})\le 1-|\alpha|$ 
\[
\left|p(a)-\frac{\alpha}{|\alpha|}\right|\le d(a,F_{p,\frac{\alpha}{|\alpha|}})\le 1-|\alpha|.
\]
In the same way we have
\[
\left|p(a)-\left(-\frac{\alpha}{|\alpha|}\right)\right|\le d(a,F_{p,-\frac{\alpha}{|\alpha|}})\le 1+|\alpha|.
\]
Then by the two inequalities  $p(a)$ have to be $\alpha$.
\end{proof}
The following is an auxiliary result.
\begin{prop}\label{prop15}
Let $B_1$ be a complex Banach space and $P_1$ a set of representatives for $\F_{B_1}$. 
Assume the following two conditions:
\begin{itemize}
\item[i)] 
$B_1$ satisfies the condition of the Hausdorff distance,
\item[ii)]
$M_{p,\alpha}=\{a\in S(B_1):p(a)=\alpha\}$ for every $p\in P_1$ and $\alpha\in \mathbb{D}$.
\end{itemize}
Then $B_1$ satisfies the complex Mazur-Ulam property.
\end{prop}
\begin{proof}
Suppose that $B_2$ is a complex Banach space and $T:B_1\to B_2$ a surjective isometry. 
Applying Lemma \ref{lemma9} for the equation \eqref{important} we get
\[
\phi(p)(T(a))=\tau(p,1)\times
\begin{cases}
p(a),\quad & p\in P_1^+ 
\\
\overline{p(a)},\quad & p\in P_1^-,
\end{cases}
\]
for $a\in S(B_1)$ such that  $|p(a)|=1$.
We prove that 
\[
\phi(p)(T(a))=\tau(p,1)\times
\begin{cases}
p(a), \quad & p\in P_1^+
\\
\overline{p(a)}, \quad & p\in P_1^-
\end{cases}
\]
for any $a\in S(B_1)$. 
Let $a\in S(B_1)$ and $\alpha=p(a)$. Then $|\alpha|\le 1$. By condition ii), $a\in M_{p,\alpha}$ and $T(a)\in T(M_{p,\alpha})$. As i) is assumed, we have by Lemma \ref{mpalpha} that
\[
T(M_{p,\alpha})=
\begin{cases}
M_{\phi(p),\alpha \tau(p,1)},\quad p\in P_1^+
\\
M_{\phi(p),\overline{\alpha}\tau(p, 1)},\quad p\in P_1^-.
\end{cases}
\]
Therefore
\[
\phi(p)(T(a))=\alpha\tau(p,1)=p(a)\tau(p,1)
\]
if $p\in P_1^+$. In a similar way we have
\[
\phi(p)(T(a))=\overline{p(a)}\tau(p,1)
\]
if $p\in P_1^-$.
We conclude that 
\[
\phi(p)(T(a))=\tau(p,1)\times
\begin{cases}
p(a), \quad &p\in P_1^+
\\
\overline{p(a)},\quad & p\in P_1^-
\end{cases}
\]
for every $a\in S(B_1)$. Let $c\in B_1$ and $c\ne 0$. Since $\frac{c}{\|c\|}\in S(B_1)$, and $\phi(p)$ and $p$ are real-linear we have 
\begin{align*}
\phi(p)(\widetilde{T}(c))& =\phi(p)\left(\|c\|T\left(\frac{c}{\|c\|}\right)\right)=\|c\|\tau(p,1)\times
\begin{cases}
p\left(\frac{c}{\|c\|}\right),\quad p\in P_1^+
\\
\overline{p\left(\frac{c}{\|c\|}\right)},\quad p\in P_1^-
\end{cases}
\\
&
=
\tau(p,1)\times
\begin{cases}
p(c), \quad &p\in P_1^+
\\
\overline{p(c)}, \quad & p\in P_1^-
\end{cases}
\end{align*}
We infer that
\begin{multline*}
\phi(p)(\widetilde{T}(a+rb))=\tau(p,1)\times 
\begin{cases}
p(a+rb), \quad p\in P_1^+ 
\\
\overline{p(a+rb)}, \quad p\in P_1^-
\end{cases}
\\
=\tau(p,1)\times 
\begin{cases}
p(a), \quad p\in P_1^+ 
\\
\overline{p(a)}, \quad p\in P_1^-
\end{cases}
+
\tau(p,1)\times
\begin{cases}
rp(b),\quad p\in P_1^+
\\
r\overline{p(b)}, \quad p\in P_1^-
\end{cases}
\\
=
\phi(p)(\widetilde{T}(a)+r\widetilde{T}(b))
\end{multline*}
for every pair $a,b\in B_1$ and every real number $r$. 
As $\phi(P_1)=P_2$ is a norming family we conclude that
\[
\widetilde{T}(a+rb)=\widetilde{T}(a)+r\widetilde{T}(b)
\]
for every pair $a,b\in $ and every real number $r$. By the definition of $\widetilde{T}$ it is a bijection from $B_1$ onto $B_2$ and it satisfies the equality $\|\widetilde{T}(a)\|=\|a\|$ for every $a\in B_1$. Thus $\widetilde{T}$ is a surjective  real-linear isometry from $B_1$ onto $B_2$ which extend $T$.
\end{proof}
The following is the main result in this paper.
\begin{theorem}\label{main}
A uniform algebra satisfies the complex Mazur-Ulam property.
\end{theorem}
\begin{proof}
Let $A$ be a uniform algebra. Put $P=\{\delta_x:x\ch(A)\}$, where $\delta_x$ is the point evaluation at $x\in \ch(A)$, the Choquet boundary. Then $P$ is a set of representatives for $\F_A$.
 It is known that a uniform algebra satisfies the condition of the Hausdorff distance \cite[Lemma 4.1]{HOST}.  By \cite[Lemma 6.3]{HOST} we have $M_{\delta_x,\alpha}=\{f\in S(A):f(x)=\alpha\}$ for every $x\in \ch(A)$ and $\alpha\in \bar\D$. Thus the conditions i) and ii) of Proposition \ref{prop15} holds for $A$. Thus $A$ satisfies the complex Mazur-Ulam property by Proposition \ref{prop15}.
 \end{proof}
\section{The case of a real Banach space}\label{sec5}
Throughout the section we denote $B_j$ a real Banach space, $P_j$ a set of representatives for $\F_{B_j}$ and $T:S(B_1)\to S(B_2)$ is a surjective isometry. 
We have by \eqref{-} that $\phi(p,1)=\phi(p,-1)$ for every $p\in P_1$. As $\T=\{\pm 1\}$ for real Banach spaces we have that $\phi(p,\lambda)$ does not depend the second term for real Banach spaces. We also have $\tau(p,-1)=-\tau(p,1)$ for every $p\in P_1$ by \eqref{-}. 
The situation is rather simple than the case of complex Banach spaces, and 
by \eqref{important} we have the following equation \eqref{rthen} without further assumption on $B_1$, i.e., 
\begin{equation}\label{rthen}
\phi(p)(T(a))=\tau(p,1)p(a)
\end{equation}
for every $p\in P_1$ and $a\in S(B_1)$ with $|p(a)|=1$.

\begin{prop}\label{prop2.2}
Suppose that 
\begin{equation}\label{mpar}
M_{p,\alpha}=\{a\in S(B_1):p(a)=\alpha\}
\end{equation}
for every $p\in P$ and $-1\le \alpha \le 1$. 
Then $T$ is extended to a surjective real-linear isometry form $B_1$ onto $B_2$.
Hence $B_1$ satisfies the Mazur-Ulam property.
\end{prop}
\begin{proof}
We first prove the equation \eqref{rthen} for every $p\in P_1$ and $a\in S(B_1)$ 
without assuming that $|p(a)|=1$.
Let $p\in P_1$ and $a\in S(B_1)$. 
Put $\alpha=p(a)$. Then by \eqref{mpar} $a\in M_{p,\alpha}$. 
We have by Lemma \ref{mpalpha} that 
\[
\phi(p)(T(a))=\alpha\tau(p,1)=\tau(p,1)p(a).
\]
It follows that for the homogeneous extension $\widetilde{T}$ of $T$ we have
\begin{multline*}
\phi(p)(\widetilde{T}(c))=\phi(p)\left(\|c\|T\left(\frac{c}{\|c\|}\right)\right)
\\
=\|c\|\tau(p,1)p\left(\frac{c}{\|c\|}\right)=\tau(p,1)p(c)
\end{multline*}
for every $0\ne c\in B_1$. As the equality $\phi(p)(\widetilde{T}(0))=\tau(p,1)p(0)$ holds, we obtain for $a,b\in B_1$ and a real number $r$ that
\[
\phi(p)(\widetilde{T}(a+rb)=\tau(p,1)p(a+rb)=\tau(p,1)p(a)+r\tau(p,1)p(b)
\]
and
\begin{multline*}
\phi(p)(\widetilde{T}(a)+r\widetilde{T}(b))=\phi(p)(\widetilde{T}(a))+r\phi(p)(\widetilde{T}(b))
\\
=\tau(p,1)p(a)+r\tau(p,1)p(b).
\end{multline*}
It follows that
\[
\phi(p)(\widetilde{T}(a+rb))=\phi(p)(\widetilde{T}(a)+r\widetilde{T}(b))
\]
for every $p\in P_1$, $a,b\in B_1$, and every real number $r$. As $\phi(P_1)=P_2$ is a norming family we see that $\widetilde{T}$ is real-linear on $B_1$. As the homogeneous extension is a norm-preserving bijection
  as is described in the first parat of section \ref{sec2} we complete the proof.
\end{proof}
\begin{definition}
For a locally compact Hausdorff space, we denote by $C_0(Y,{\mathbb{R}})$ a real Banach space of all real-valued continuous functions which vanish at infinity on $Y$. 
Let $E$ be a closed subspace of $C_0(Y,\mathbb{R})$ which separates the points of $Y$, that is, for any pair $y_1$ and $y_2$ of different points in $Y$ there exists a function $e\in E$ such that $e(y_1)\ne e(y_2)$.  In this paper we say that $E$ satisfies the condition $(r)$ if for any triple $y\in \ch(E)$, a neighborhood $V$, and $\varepsilon >0$ there exists $u\in E$ such that $0\le u\le 1=u(y)$ on $Y$ and $0\le u\le \varepsilon$ on $Y\setminus V$. 
\end{definition}
Note that if $E$ satisfies the condition $(r)$, then it is extremely $C$-regular (cf. \cite[Definition 2.3.9]{fj1}).

\begin{example}\label{er}
The space $C_0(Y,\R)$ satisfies the condition $(r)$ for any locally compact Hausdorff space $Y$. Let 
\[
E=\{f\in C_0((0,2],\mathbb{R}): \text{$f(t)=at$ on $(0,1]$ for some $a\in \R$}\}.
\]
Then $\ch(E)=[1,2]$ and $E$ satisfies the condition $(r)$.
\end{example}
Liu \cite[Corollary 6]{liu2007} established the Mazur-Ulam property for $C_0(Y,\R)$ when $Y$ is compact. The following generalizes the result of Liu.
\begin{cor}\label{cor2.3}
Let $Y$ be a locally compact Hausdorff space. 
Suppose that $E$ is a closed subspace of $C_0(Y,\R)$ which separates the  points of $Y$. Suppose that $E$ satisfies the condition $(r)$. Then $E$ satisfies the Mazur-Ulam property. In particular, $C_0(Y,\R)$ satisfies the Mazur-Ulam property.
\end{cor}

\begin{proof}
Let $\ch(E)$ be the Choquet boundary for $E$. Then  by the Arens-Kelley theorem $P=\{\delta_x: x\in \ch(E)\}$ is a set of representatives for $\F_E$. Suppose that $p\in P$ and $-1\le \alpha\le 1$.
 Proving the inclusion 
\[
M_{p,\alpha}\supset \{a\in S(E):p(a)=\alpha\},
\]
we get \eqref{mpar} by Lemma \ref{lemma}. It will follw by Proposition \ref{prop2.2} that $E$ satisfies the Mazur-Ulam property. Let $a\in S(B_1)$ such that $p(a)=\alpha$. 
First we consider the case of $\alpha=1$. 
By definition $a\in F_{p,1}$, hence $d(a, F_{p,1})=0=1-|\alpha|$.  For every $b\in F_{p,-1}$ we have $2=|p(a)-p(b)|\le \|a-b\|\le 2$. Hence $d(a,F_{p,-1})=2=1+|\alpha|$. Thus $a\in M_{p,\alpha}$ if $\alpha=1$. A proof for $\alpha=-1$ is similar and is omitted. 

We consider the case that $-1<\alpha<1$. Let $a\in S(B_1)$ such that $p(a)=\alpha$. Let $\varepsilon$ be $0<\varepsilon<1-|\alpha|$. Put $K=\{q\in P:q(a)=\alpha\}$  and $F_0=\{q\in P:|q(a)-\alpha|\ge\varepsilon/4\}$. For each positive integer $n$, put $F_n=\{q\in P: \varepsilon/2^{n+2}\le |q(a)-\alpha|\le \varepsilon/2^{n+1}\}$.  Then $P=K\cup\left(\cup_{m=0}^\infty F_m\right)$. For each positive integer $n$, choose $u_n\in S(B_1)$ such that $0\le u_n\le 1=p(u_n)$ on $Y$ and $u_n=0$ on $F_0\cup F_n$. Since $P=\{\delta_x:x\in \ch(E)\}$ and $E$ satisfies the condition $(r)$ such $u_n\in E$ exists. 
Put $u=\sum_{n=1}^\infty u_n/2^n$. 
As the supremum norm of $u_n$ is dominated by $1$ for every $n$, $u$ converges uniformly on $Y$ in $E$.
Then put
\[
g_+=\left(
\frac{\alpha}{|\alpha|}-\alpha\right)u+a,
\]
and
\[
g_-=\left(
-\frac{\alpha}{|\alpha|}-\alpha\right)u+a.
\]
We see that $g_+\in F_{p,\frac{\alpha}{|\alpha|}}$, $\|g_+-a\|=1-|\alpha|$, and  $d(a, F_{p,\frac{\alpha}{|\alpha|}})\le1-|\alpha|$. A proof is as follows.
All what is really needed is to prove that $|q(g_-)|\le 1$ for every $q\in P$ as
it is evident that $g_+\in E$ and  $p(g_+)=\frac{\alpha}{|\alpha|}$ by the definition of $g_+$. 
 If $q\in F_0$, then $q(u)=0$ asserts that $|q(g_+)|=|q(a)|\le 1$.
Suppose that $q\in F_n$ for some positive integer $n$. Then $q(u_n)=0$ and $0\le u_k\le 1$ on $Y$ we have $|q(u)|\le \sum_{k\ne n}1/2^k=1-1/2^n$. Thus  
\begin{align*}
|q(g_+)| & \le \left|\frac{\alpha}{|\alpha|}-\alpha\right|(1-1/2^n)+|q(a)-\alpha|+|\alpha|
\\
&\le (1-|\alpha|)(1-1/2^n)+\varepsilon/2^{n+1}+|\alpha|
\\
&\le (1-|\alpha|)(1-1/2^n)+(1-|\alpha|)/2^{n+1}+|\alpha|\le 1
\end{align*}
Suppose that $q\in K$. Then 
\[
|q(g_+)|\le 1-|\alpha|+|\alpha|=1.
\]
We conclude that $g_+\in F_{p,\frac{\alpha}{|\alpha|}}$. By the definition of $g_+$, we infere that 
$\|g_+-a\|=(1-|\alpha|)\|u\|=1-|\alpha|$.
It follows that 
\begin{equation}\label{a+}
d(a, F_{p,\frac{\alpha}{|\alpha|}})\le1-|\alpha|.
\end{equation}

We also see that $g_-\in F_{p,-\frac{\alpha}{|\alpha|}}$, $\|g_--a\|=1+|\alpha|$ and $d(a,F_{p,-\frac{\alpha}{|\alpha|}})\le1+|\alpha|$. A proof is as follows. 
All what is really needed is to prove that $|q(g_-)|\le 1$ for every $q\in P$ as
it is evident that $g_-\in E$ and $p(g_-)=-\frac{\alpha}{|\alpha|}$.
 If $q\in F_0$, then $q(u)=0$ asserts that $|q(g_+)|=|q(a)|\le 1$.
 Suppose that $q\in F_n$ for some positive integer $n$. 
 We have
 \[
 q(g_-)=\left(-\frac{\alpha}{|\alpha|}+\alpha\right)q(u)-2\alpha q(u)+\alpha +q(a)-\alpha,
 \]
 hence
\[
 |q(g_-)|\le (1-|\alpha|)|q(u)|+|\alpha||1-2q(u)|+|q(a)-\alpha|.
 \]
 As $0\le u\le1$ on $Y$ by the definition of $u$, we have $|1-2q(u)|\le 1$ for every $q\in P=\{\delta_x:x\in \ch(E)\}$. As we have shown that $|q(u)|\le 1-1/2^n$ for $q\in F_n$
 \begin{align*}
 |q(g_-)&|\le (1-|\alpha|)|q(u)|+|\alpha|+|q(a)-\alpha|
 \\
 & \le (1-|\alpha|)(1-1/2^n)+|\alpha|+\varepsilon/2^{n+1}
 \\
 & \le(1-|\alpha|)(1-1/2^n)+|\alpha|+(1-|\alpha|)/2^{n+1}\le 1.
 \end{align*}
Suppose that $q\in K$.
Then we have
\[
q(g_-)=\left(-\frac{\alpha}{|\alpha|}+\alpha\right)q(u)-2\alpha q(u)+\alpha.
\]
As $0\le u\le 1$ on $Y$ we have
\[
|q(g_-)|\le 1-|\alpha|+|\alpha||1-2q(u)|\le 1.
\]
We conclude that $g_-\in F_{p,-\frac{\alpha}{|\alpha|}}$. By the definition of $g_-$, we infer that 
$\|g_--a\|=(1+|\alpha|)\|u\|=1+|\alpha|$.
It follows that 
\begin{equation}\label{a-}
d(a, F_{p,-\frac{\alpha}{|\alpha|}})\le1+|\alpha|.
\end{equation}
By \eqref{a+} and \eqref{a-} we conclude that $a\in M_{p,\alpha}$.

Finally, as $C_0(Y,\R)$ satisfies the condition $(r)$ by Urysohn's lemma, we see that 
$C_0(Y,\R)$ satisfies the Mazur-Ulam property.
\end{proof}
\begin{example}\label{examp2.4}
Let $\ell^\infty(\Gamma,\R)$ be the real Banach space of all real-valued bounded functions on a discrete space $\Gamma$. Then there is a compact Hausdorff space $X$ such that $\ell^\infty(\Gamma,\R)$ is isometrically isomorphic to $C_0(X,\R)$ as a real Banach space. 
Therefore $\ell^\infty(\Gamma,\R)$ satisfies the Mazur-Ulam property  (cf. \cite{ding2007,liu2007}). The Mazur-Ulam property of the space $c_0(\Gamma,\R)$  is established by Ding\cite[Corollary 2]{ding2007}. As $c_0(\Gamma,\R)$ is isometrically isomorphic to $C_0(Y,\R)$ for a locally compact Hausdorff space $Y$, Corollary \ref{cor2.3} gives an alternative proof of the result of Ding.
\end{example}

\begin{example}\label{examp2.5}
Let $\mu$ be a positive measure on a $\sigma$-algebra $\Sigma$ of a subsets of a set $\Omega$. Let $L^\infty(\Omega, \Sigma,\mu, \R)$ be the usual real Banach space of all real-valued bounded measurable functions on $(\Omega,\Sigma,\mu)$, Then there exists a compact Hausdorff space $X$ such that $L^\infty(\Omega,\Sigma,\mu,\R)$ is isometrically isomorphic to $C_0(X,\R)$ as a real Banach space. Hence the space $L^\infty(\Omega,\Sigma,\mu,\R)$ satisfies the Mazur-Ulam property (cf. \cite[Corollary 6]{liu2007}). Note that 
Tan \cite[Theorem 2.5]{tan2011a} exhibits  the result for the case of the measure being $\sigma$-finite by an alternative proof.
\end{example}

\section{Remarks}\label{sec6}
We close the paper with a few remarks. 
In this paper we merely prove the complex Mazur-Ulam property for a uniform algebra. We conjecture that a uniform algebra, generally a closed subalgebra of $C_0(Y,\C)$, satisfies the Mazur-Ulam property. 

The  second remark concerns Tingley's problem on a Banach space of analytic functions. The main result of this paper concerns with the complex Mazur-Ulam property. We expect that several Banach space of analytic functions including the Hardy spaces satisfy the complex Mazur-Ulam property and the Mazur-Ulam property.

As a final remark we encourage researches on Tingley's problem on Banach algebras of continuous functions. Comparing with the theorem of Wang \cite{wang1996a} on the Banach algebra of $C^{(n)}(X)$, it is interesting to study a surjective isometry on the unit sphere of a Banach space or algebra of Lipschitz functions, it has already been pointed out by Cueto-Avellaneda \cite[Problem 4.0.8]{cueto}. 

\subsection*{Acknowledgments}
The author was supported by JSPS KAKENHI Grant Numbers JP19K03536.  


\begin{thebibliography}{99}
\bibitem{banakh}
T.~Banakh,
\emph{
Every 2-dimensional Banach space has the Mazur-Ulam property},
preprint


\bibitem{bcfp}
J.~Becerra-Guerrero, M.~Cueto-Avellaneda, F.~J.~Fern\'andez-Polo and A.~M.~Peralta,
\emph{
On the extensin of isometries between the unit spheres of a JBW*-triple and a Banach space},
to appear in J. Inst. Math. Jussieu 
doi:10.1007/s13324-022-00448-2

\bibitem{br}
A.~Browder,
\emph{
Introduction to function algebras},
W. A. Benjamin, Inc., New York-Amsterdam 1969 Xii+273 pp

\bibitem{cabellosanchez2019}
J.~Cabello S\'anchez,
\emph{
A reflection on Tingley's problem and some applications},
J. Math. Anal. Appl. {\bf476} (2019), 319--336
doi:10.1016/j.jmaa.2019.03.041

\bibitem{chengdong}
L.~Cheng and Y.~Dong,
\emph{
On a generalized Mazur-Ulam question: extension of isometries between unit spheres of Banach spaces},
J. Math. Anal. Appl. {\bf 377} (2011), 464--470
doi:10.1016/j.jmaa.2020.11.025

\bibitem{cueto}
M.~Cueto-Avellaneda,
\emph{
Extension of isometreis and the Mazur-Ulam property},
PhD thesis, Universidad de Almer\'ia,  2020


\bibitem{cp2019}
M.~Cueto-Avellaneda and A.~M.~Peralta,
\emph{
On the Mazur-Ulam property for the space of 
Hilbert-space-valued continuous functions},
J. Math. Anal. Appl. {\bf 479}(2019), 875--902
doi:10.1016/j.jmaa.2019.06.056

\bibitem{cp2020}
M.~Cueto-Avellaneda and A.~M.~Peralta,
\emph{
The Mazur-Ulam property for commutative von Neumann algebras},
Linear Multlinear Algebra {\bf68} (2020), 337--362
doi:10.1080/03081087.2018.1505823


\bibitem{ding2003B}
G.~G.~Ding,
\emph{
On extension of isometries between unit spheres of $E$ and $C(\Omega)$},
Acta Math. Sin. (Engl. Ser.)  {\bf19} (2003), 793--800

\bibitem{ding2007}
G.~G.~Ding,
\emph{
The isometric extension of the into mapping from a ${\mathcal L}^\infty(\Gamma)$-type space to some Banach space},
Illinois J. Math. {\bf 51} (2007), 445--453



\bibitem{fj1}
R.~J.~Fleming and J.~E.~Jamison,
\emph{
Isometries on Banach spaces: function spaces},
 Chapman \& Hall/CRC Monographs and Surveys in Pure and Applied Mathematics, 129. Chapman \& Hall/CRC, Boca Raton, FL, 2003. x+197 pp. ISBN: 1-58488-040-6

\bibitem{HOST}
O.~Hatori, S.~Oi and R.~Shindo Togashi,
\emph{
}
to appear in 
Jour. Math. Anal. Appl. 
doi: 10.1016/j.jmaa.2021.125346

\bibitem{jmpr2019}
 A.~Jim\'enez-Vargas, A.~Morales-Campoy, A.~M.~Peralta and M.~I.~Ram\'irez,
 \emph{
 The Mazur-Ulam property for the space of complex null sequences},
 Linear Multilinear Algebra {\bf 67} (2019), 799--816
 doi:10.1080/03081087.2018.1433625

\bibitem{kp}
 O.~F.~K.~Kalenda and A.~M.~Peralta,
 \emph{
 Extension of isometries from the unit sphere of a rank-2 Cartan factor},
 Anal. Math. Phys. {\bf 11}, Article number:15 (2021)
 
\bibitem{liu2007}
 R.~Liu,
 \emph{
 On extension of isometries between unit spheres of ${\mathcal L}^\infty(\Gamma)$-type space and a Banach space $E$},
 J. Math. Anal. Appl. {\bf333} (2007), 959--970
 doi:10.1016/j.jmaa.2006.11.044
 

\bibitem{mori}
M.~Mori,
\emph{
Tingley's problem through the facial structure of operator algebras},
J. Math. Anal. Appl. {\bf 466} (2018), 1281--1298
doi:10.1016/j.jmaa.2018.06.050

\bibitem{moriozawa}
M.~Mori and N.~Ozawa,
\emph{
Mankiewicz's theorem and the Mazur-Ulam property for $C^*$-algebras},
Studia Math. {\bf 250} (2020), 265--281
doi:10.4064/sm180727-14-11

\bibitem{peralta2019a}
A.~M.~Peralta,
\emph{
Extending surjective isometries defined on the unit sphere of $\ell_\infty(\Gamma)$}
Rev. Mat. Complut. {\bf32} (2019), 99--114
doi:10.1007/s13163-018-0269-2

\bibitem{tan2011a}
D.~N.~Tan,
\emph{
Extension of isometries on unit spheres of $L^\infty$},
Taiwanese J. Math. {\bf15} (2011), 819--827

\bibitem{tan2011b}
D.~N.~Tan,
\emph{
On extension of isometries on the unit spheres of $L^p$-spaces for $0<p\le 1$},
Nonlinear Anal. {\bf74} (2011), 6981--6987
doi:10.1016/j.na.2011.07.035

\bibitem{tan2012a}
D.~N.~Tan,
\emph{
Extension of isometries on the unit sphere of $L^p$ spaces},
Acta Math. Sin. (Engl. Ser.) {\bf28} (2012), 1197--1208
doi:10.1007/s10114-011-0302-6


\bibitem{thl2013}
D.~Tan, X.~Huang and R.~Liu,
\emph{
Generalized-lush spaces and the Mazur-Ulam property},
Studia Math. {\bf219} (2013), 139--153
doi:10.4064/sm219-2-4

\bibitem{tanaka2014b}
R.~Tanaka,
\emph{
A further property of spherical siometries},
Bull. Aust. Math. Soc. {\bf 90} (2014), 304--310
doi:10.1017/S0004972714000185

\bibitem{tanaka}
R.~Tanaka,
\emph{
The solution of Tingley's problem for the operator norm unit sphere of complex $n\times n$ matrices},
Linear Algebra Appl. {\bf 494} (2016), 274--285
doi:10.1016/j.laa.2016.01.020

\bibitem{tingley}
D.~Tingley,
\emph{
Isometries of the unit sphere},
Geom. Dedicata {\bf 22} (1987), 371--378

\bibitem{wang}
R.~S.~Wang,
\emph{
Isometries between the unit spheres of $C_0(\Omega)$ type spaces},
Acta Math. Sci. (English Ed.) {\bf 14} (1994), 82--89

\bibitem{wang1996a}
Risheng Wang,
\emph{
Isometries of $C_0^{(n)}(X)$},
Hokkaido Math. J. {\bf 25} (1996), 465--519
doi:10.14492/hokmj/1351516747

\bibitem{wh2019}
Ruidong Wang and X.~Huang,
\emph{
The Mazur-Ulam property for two dimensional somewhere-flat spaces},
Linear Algebra Appl. {\bf562} (2019), 55--62
doi:10.1016/j.laa.2018.09.024


\bibitem{yangzhao}
X.~Yang and X.~Zhao,
\emph{
On the extension problems of isometric and nonexpansive mappings},
In:Mathematics without boundaries. Edited by Themistocles M. Rassias and Panos M. Pardalos, 725-- Springer, New York, 2014






\end{thebibliography}
\end{document}